\documentclass[a4paper,10pt]{amsart}
\usepackage{fourier}
\usepackage[margin=2.75cm]{geometry}
\usepackage{mathrsfs}
\usepackage{stmaryrd}
\usepackage{tikz}
\usetikzlibrary{arrows,%
  decorations.markings,%
  matrix,%
  shapes}

\newcommand{\A}{\mathcal{A}}
\newcommand{\C}{\mathcal{C}}
\newcommand{\T}{\mathcal{T}}

\newcommand{\ceq}{\mathrel{\mathop:}=}

\newcommand{\nang}{\mathscr{N}}
\newcommand{\susp}{\widehat{\Sigma}}
\DeclareMathOperator{\complete}{Comp}
\DeclareMathOperator{\dense}{Dense}

\DeclareMathOperator{\Hom}{Hom}

\DeclareMathOperator{\Image}{Im}

\DeclareMathOperator{\subgroup}{Sub}
\DeclareMathOperator{\ideal}{Ideal}
\DeclareMathOperator{\Prime}{Prime}

\theoremstyle{definition}
\newtheorem{definition}[subsection]{Definition}
\newtheorem{remark}[subsection]{Remark}

\theoremstyle{plain}
\newtheorem{theorem}[subsection]{Theorem}
\newtheorem{lemma}[subsection]{Lemma}
\newtheorem{proposition}[subsection]{Proposition}
\newtheorem{corollary}[subsection]{Corollary}

\title{The Grothendieck group of an $n$-angulated category}
\author{Petter Andreas Bergh}
\author{Marius Thaule}
\address{Department of Mathematical Sciences, NTNU, NO-7491
  Trondheim, Norway}
\email{bergh@math.ntnu.no}
\email{mariusth@math.ntnu.no}

\begin{document}

\begin{abstract}
  We define the Grothendieck group of an $n$-angulated category and
  show that for odd $n$ its properties are as in the special case of
  $n = 3$, i.e.\ the triangulated case.  In particular, its subgroups
  classify the dense and complete $n$-angulated subcategories via a
  bijective correspondence.  For a tensor $n$-angulated category, the
  Grothendieck group becomes a ring, whose ideals classify the dense
  and complete $n$-angulated tensor ideals of the category.
\end{abstract}

\thanks{The second author was partially supported by the Research
  Council of Norway grant Topology, grant number 185335/V30.}


\maketitle

\section{Introduction}
\label{sec:intro}
The Grothendieck group of a triangulated category is the free abelian
group on the (set of) isomorphism classes of objects, modulo the Euler
relations corresponding to the distinguished triangles.  Thomason
showed in \cite{Thomason} that the set of subgroups of the
Grothendieck group classifies the dense triangulated subcategories.
Namely, there is a bijective correspondence between the set of
subgroups and the set of dense triangulated subcategories.

Recently, in \cite{GeissKellerOppermann}, Geiss, Keller and Oppermann
introduced ``higher dimensional'' analogues of triangulated
categories, called \emph{$n$-angulated categories}.  An $n$-angulated
category with $n = 3$ is nothing but a classical triangulated
category, and they showed that certain cluster tilting subcategories
of triangulated categories give rise to higher $n$-angulated
categories.

In this paper, we define and study the Grothendieck group of an
$n$-angulated category.  As in the triangulated case, it is the free
abelian group on the set of isomorphism classes of objects, modulo the
Euler relations corresponding to the $n$-angles.  Our main result
shows that when $n$ is odd, then the set of subgroups corresponds
bijectively to the complete and dense $n$-angulated subcategories,
thus providing a classification of these.  When $n = 3$, that is, in
the classical triangulated case, every triangulated subcategory is
complete, hence we recover Thomason's classification theorem. Our
proof of the classification result only works when $n$ is odd; we do
not know whether a similar result holds when $n$ is an even integer.
The underlying reason for this is that in the odd case, the additive
inverse of an element in the Grothendieck group is given by the
suspension of the corresponding object.  This is no longer true in the
even case.

We also define \emph{tensor $n$-angulated categories}, that is,
$n$-angulated categories with a symmetric tensor structure (or
symmetric monoidal structure) which is compatible with the
$n$-angulation.  For such a category, the Grothendieck group becomes a
ring in a natural way, the \emph{Grothendieck ring}.  We show that the
set of ideals in this ring classify the complete and dense
$n$-angulated tensor ideals of the category.

This paper is organized as follows.  In Section \ref{sec:gdieck}, we
recall the definition of an $n$-angulated category, define its
Grothendieck group, and prove some elementary properties.  In Section
\ref{sec:cluster}, we specialize to the case of an $n$-angulated
category arising from a cluster tilting subcategory of a triangulated
category.  We prove that there is a natural surjective homomorphism
from the Grothendieck group of the $n$-angulated category to the
Grothendieck group of the triangulated category.  In Section
\ref{sec:classifying}, we prove our main result, the classification
theorem which gives a bijective correspondence between subgroups and
complete and dense $n$-angulated subcategories.  Finally, in Section
\ref{sec:tensor} we define tensor $n$-angulated categories, and prove
the tensor version of the classification theorem.

\section{The Grothendieck group of an $n$-angulated category}
\label{sec:gdieck}
Throughout this paper, every category will be assumed to be small,
that is, the collection of isomorphism classes of objects forms a set.

We start by recalling the definition of an $n$-angulated category from
\cite{GeissKellerOppermann}.  Let $\C$ be an additive category with an
automorphism $\Sigma\colon \C \to \C$, and $n$ an integer greater than
or equal to three.  An \emph{$n$-$\Sigma$-sequence} in $\C$ is a
sequence
\begin{equation*}
  A_1 \xrightarrow{\alpha_1} A_2 \xrightarrow{\alpha_2} \cdots
  \xrightarrow{\alpha_{n - 1}} A_n \xrightarrow{\alpha_n} \Sigma A_1 
\end{equation*}
of objects and morphisms in $\C$.  We shall often denote such sequences
by $A_\bullet, B_\bullet$ etc.  Its left and right \emph{rotations} are
the two $n$-$\Sigma$-sequences
\begin{equation*}
  A_2 \xrightarrow{\alpha_2} A_3 \xrightarrow{\alpha_3} \cdots
  \xrightarrow{\alpha_n} \Sigma A_1 \xrightarrow{(-1)^n
    \Sigma\alpha_1} \Sigma A_2
\end{equation*}
and
\begin{equation*}
  \Sigma^{-1}A_n \xrightarrow{(-1)^n \Sigma^{-1} \alpha_n} A_1
  \xrightarrow{\alpha_1} \cdots \xrightarrow{\alpha_{n - 2}} A_{n-1}
  \xrightarrow{\alpha_{n - 1}} A_n
\end{equation*} 
respectively, and it is \emph{exact} if the induced sequence
\begin{equation*}
  \cdots \to \Hom_{\C}(B,A_1) \xrightarrow{(\alpha_1)_*}
  \Hom_{\C}(B,A_2) \xrightarrow{(\alpha_2)_*} \cdots
  \xrightarrow{(\alpha_{n-1})_*} \Hom_{\C}(B,A_n)
  \xrightarrow{(\alpha_n)_*} \Hom_{\C}(B, \Sigma A_1) \to \cdots
\end{equation*}
of abelian groups is exact for every object $B \in \C$. A
\emph{trivial} $n$-$\Sigma$-sequence is a sequence of the form
\begin{equation*}
  A \xrightarrow{1} A \to 0 \to \cdots \to 0 \to \Sigma A
\end{equation*}
or any of its rotations.  A \emph{morphism} $A_{\bullet}
\xrightarrow{\varphi} B_{\bullet}$ of $n$-$\Sigma$-sequences is a
sequence $\varphi = (\varphi_1,\varphi_2,\ldots,\varphi_n)$ of
morphisms in $\C$ such that the diagram

\begin{center}
  \begin{tikzpicture}
    \node (X1) at (0,1.5){$A_1$};
    \node (X2) at (1.5,1.5){$A_2$};
    \node (X3) at (3,1.5){$A_3$};
    \node (Xdots) at (4.5,1.5){$\cdots$};
    \node (Xn) at (6,1.5){$A_n$};
    \node (SX1) at (7.5,1.5){$\Sigma A_1$};

    \node (Y1) at (0,0){$B_1$};
    \node (Y2) at (1.5,0){$B_2$};
    \node (Y3) at (3,0){$B_3$};
    \node (Ydots) at (4.5,0){$\cdots$};
    \node (Yn) at (6,0){$B_n$};
    \node (SY1) at (7.5,0){$\Sigma B_1$};

    \begin{scope}[->,font=\scriptsize,midway]
      \draw (X1) -- node[right]{$\varphi_1$} (Y1);
      \draw (X2) -- node[right]{$\varphi_2$} (Y2);
      \draw (X3) -- node[right]{$\varphi_3$} (Y3);
      \draw (Xn) -- node[right]{$\varphi_n$} (Yn);
      \draw (SX1) -- node[right]{$\Sigma \varphi_1$} (SY1);

      \draw (X1) -- node[above]{$\alpha_1$} (X2);
      \draw (X2) -- node[above]{$\alpha_2$} (X3);
      \draw (X3) -- node[above]{$\alpha_3$} (Xdots);
      \draw (Xdots) -- node[above]{$\alpha_{n - 1}$} (Xn);
      \draw (Xn) -- node[above]{$\alpha_n$} (SX1);
      \draw (Y1) -- node[above]{$\beta_1$} (Y2);
      \draw (Y2) -- node[above]{$\beta_2$} (Y3);
      \draw (Y3) -- node[above]{$\beta_3$} (Ydots);
      \draw (Ydots) -- node[above]{$\beta_{n - 1}$} (Yn);
      \draw (Yn) -- node[above]{$\beta_n$} (SY1);
    \end{scope}
  \end{tikzpicture}
\end{center}
commutes.  It is an \emph{isomorphism} if
$\varphi_1,\varphi_2,\ldots,\varphi_n$ are all isomorphisms in $\C$,
and a \emph{weak isomorphism} if $\varphi_i$ and $\varphi_{i+1}$ are
isomorphisms for some $i$ (with $\varphi_{n+1} \ceq \Sigma \varphi_1$).

The category $\C$ is \emph{pre-$n$-angulated} if there exists a
collection $\nang$ of $n$-$\Sigma$-sequences satisfying the following
three axioms:
\begin{itemize}
\item[{\textbf{(N1)}}]
  \begin{itemize}
  \item[(a)] $\nang$ is closed under direct sums, direct summands and
    isomorphisms of $n$-$\Sigma$-sequences.
  \item[(b)] For all $A \in \C$, the trivial $n$-$\Sigma$-sequence
   \begin{equation*}
     A \xrightarrow{1} A \to 0 \to \cdots \to 0 \to \Sigma A
   \end{equation*}
   belongs to $\nang$.
 \item[(c)] For each morphism $\alpha \colon A_1 \to A_2$ in $\C$,
   there exists an $n$-$\Sigma$-sequence in $\nang$ whose first
   morphism is $\alpha$.
 \end{itemize}
\item[{\textbf{(N2)}}] An $n$-$\Sigma$-sequence belongs to $\nang$ if
  and only if its left rotation belongs to $\nang$.
\item[{\textbf{(N3)}}] Each commutative diagram
  \begin{center}
    \begin{tikzpicture}[text centered]
      \node (X1) at (0,1.5){$A_1$};
      \node (X2) at (1.5,1.5){$A_2$};
      \node (X3) at (3,1.5){$A_3$};
      \node (Xdots) at (4.5,1.5){$\cdots$};
      \node (Xn) at (6,1.5){$A_n$};
      \node (SX1) at (7.5,1.5){$\Sigma A_1$};

      \node (Y1) at (0,0){$B_1$};
      \node (Y2) at (1.5,0){$B_2$};
      \node (Y3) at (3,0){$B_3$};
      \node (Ydots) at (4.5,0){$\cdots$};
      \node (Yn) at (6,0){$B_n$};
      \node (SY1) at (7.5,0){$\Sigma B_1$};

      \begin{scope}[font=\scriptsize,->,midway]
        \draw (X1) -- node[right]{$\varphi_1$} (Y1);
        \draw (X2) -- node[right]{$\varphi_2$} (Y2);
        \draw[dashed] (X3) -- node[right]{$\varphi_3$} (Y3);
        \draw[dashed] (Xn) -- node[right]{$\varphi_n$} (Yn);
        \draw (SX1) -- node[right]{$\Sigma \varphi_1$} (SY1);
      
        \draw (X1) -- node[above]{$\alpha_1$} (X2);
        \draw (X2) -- node[above]{$\alpha_2$} (X3);
        \draw (X3) -- node[above]{$\alpha_3$} (Xdots);
        \draw (Xdots) -- node[above]{$\alpha_{n - 1}$} (Xn);
        \draw (Xn) -- node[above]{$\alpha_n$} (SX1);
        \draw (Y1) -- node[above]{$\beta_1$} (Y2);
        \draw (Y2) -- node[above]{$\beta_2$} (Y3);
        \draw (Y3) -- node[above]{$\beta_3$} (Ydots);
        \draw (Ydots) -- node[above]{$\beta_{n - 1}$} (Yn);
        \draw (Yn) -- node[above]{$\beta_n$} (SY1);
      \end{scope}
    \end{tikzpicture}
  \end{center}
  with rows in $\nang$ can be completed to a morphism of
  $n$-$\Sigma$-sequences.
\end{itemize}

In this case, the collection $\nang$ is a \emph{pre-$n$-angulation} of
the category $\C$ (relative to the automorphism $\Sigma$), and the
$n$-$\Sigma$-sequences in $\nang$ are \emph{$n$-angles}.  If, in
addition, the collection $\nang$ satisfies the following axiom, then
it is an \emph{$n$-angulation} of $\C$, and the category is
\emph{$n$-angulated}:

\begin{itemize}
\item[{\textbf{(N4)}}] In the situation of (N3), the morphisms
  $\varphi_3,\varphi_4,\ldots,\varphi_n$ can be chosen such that the
  mapping cone
  \begin{equation*}
    A_2 \oplus B_1 \xrightarrow{\left[
        \begin{smallmatrix}
          -\alpha_2 & 0\\
          \hfill \varphi_2 & \beta_1
        \end{smallmatrix}
      \right]} A_3 \oplus B_2 \xrightarrow{\left[
        \begin{smallmatrix}
          -\alpha_3 & 0\\
          \hfill \varphi_3 & \beta_2
        \end{smallmatrix}
      \right]} \cdots \xrightarrow{\left[
        \begin{smallmatrix} 
          -\alpha_n & 0\\
          \hfill \varphi_n & \beta_{n - 1}
        \end{smallmatrix}
      \right]} \Sigma A_1 \oplus B_n \xrightarrow{\left[
       \begin{smallmatrix}
          -\Sigma \alpha_1 & 0\\
          \hfill \Sigma \varphi_1 & \beta_n
        \end{smallmatrix}
      \right]} \Sigma A_2 \oplus \Sigma B_1
  \end{equation*}
belongs to $\nang$.
\end{itemize}

Note that the axioms given here are the original ones presented by
Geiss, Keller and Oppermann in \cite{GeissKellerOppermann}.  In
\cite{BerghThaule}, the authors gave a set of alternative axioms and
showed that they are equivalent to the original ones.  In particular,
it was shown that axiom (N4) is equivalent to a ``higher'' version of
Verdier's original octahedral axiom.

The construction and properties of Grothendieck groups do not require
axiom (N4).  Therefore, the theory we present in this paper is valid
for pre-$n$-angulated categories.  However, we have chosen to state
the definitions and results for $n$-angulated categories.

Having recalled the definition of an $n$-angulated category, we now
define the Grothendieck group.  As in the triangulated case, it is the
free abelian group on the set of isomorphism classes of objects modulo
the Euler relations corresponding to the $n$-angles.  It will be
convenient to have a shorter notation for these Euler relations.
Suppose therefore that $\C$ is an $n$-angulated category, and let
$F(\C)$ be the free abelian group on the set of isomorphism classes
$\langle A \rangle$ of objects $A$ in $\C$.  Given an $n$-angle
\begin{equation*}
  A_\bullet \colon \quad A_1 \to A_2 \to A_3 \to \cdots \to A_n \to
  \Sigma A_1
\end{equation*}
in $\C$, we denote the corresponding Euler relation in $F(\C)$ by
$\chi(A_\bullet)$, i.e.\
\begin{equation*}
  \chi(A_\bullet) \ceq \langle A_1 \rangle - \langle A_2 \rangle +
  \langle A_3 \rangle - \cdots + (-1)^{n + 1} \langle A_n \rangle.
\end{equation*}
These generate the relations in the Grothendieck group when $n$ is odd.
When $n$ is even, we also include the trivial relation $\langle 0
\rangle$, in order for the residue class of $\langle 0 \rangle$ to be
the zero element; see Proposition \ref{prop:elementary-properties}(1).

\begin{definition}
  \label{def:K0}
  Let $(\C,\Sigma)$ be an $n$-angulated category, and $F(\C)$
  the free abelian group on the set of isomorphism classes $\langle A
  \rangle$ of objects $A$ in $\C$.  Furthermore, let $R(\C)$ be the
  subgroup of $F(\C)$ generated by the following sets of elements
  \begin{equation*}
    \begin{array}{rl}
      \{\chi(A_\bullet) \mid A_\bullet \text{ $n$-angle in } \C \} &
      \text{if $n$ is odd}\\
      \{\langle 0 \rangle \} \cup \{\chi(A_\bullet) \mid A_\bullet
      \text{ $n$-angle in } \C\} & \text{if $n$ is even}
    \end{array}
  \end{equation*}
  in $\C$.  The \emph{Grothendieck group} $K_0(\C)$ of $\C$ is the
  quotient group $F(\C)/R(\C)$.  Given an object $A \in \C$, the
  residue class $\langle A \rangle + R(\C)$ in $K_0(\C)$ is denoted by
  $[A]$.
\end{definition}

Note that the definition of the Grothendieck group is the reason why
we are only considering small categories: the collection of
isomorphism classes in the category must form a set.  Note also that
it follows immediately from the definition that the Grothendieck group
of $\C$ is universal with respect to group homomorphisms from $F(\C)$
to abelian groups satisfying the Euler relations.  To be precise, let
$G$ be an abelian group and $f \colon F(\C) \to G$ a homomorphism such
that $f\left(\chi(A_\bullet)\right) = 0$ for all $n$-angles
$A_\bullet$ in $\C$, and such that $f( \langle 0 \rangle )=0$ when $n$
is even.  Then there exists a unique homomorphism $\hat{f} \colon
K_0(\C) \to G$ such that the diagram
\begin{center}
  \begin{tikzpicture}
    \node (X1) at (0,1.5){$F(\C)$};
    \node (X2) at (4,1.5){$G$};
    
    \node (Y1) at (2,0){$K_0(\C)$};
    
    \begin{scope}[font=\scriptsize,->,midway]
      \draw (X1) -- node[above]{$f$} (X2);
      \draw (X1.south east) -- node[above right]{$\pi$} (Y1.north
        west);
      \draw (Y1.north east) -- node[above left]{$\hat{f}$} (X2.south
        west);
    \end{scope}
  \end{tikzpicture}
\end{center}
commutes, where $\pi \colon F(\C) \to K_0(\C)$ is the natural
projection.

We now prove some elementary properties of the Grothendieck group,
properties which are well-known when $n = 3$, that is, when the
category is triangulated.  Note that for an arbitrary $n$, the
relation $[\Sigma A] = -[A]$ holds when $n$ is odd, whereas $[\Sigma
A] = [A]$ when $n$ is even.

\begin{proposition}
  \label{prop:elementary-properties}
  Let $(\C,\Sigma)$ be an $n$-angulated category, and $K_0(\C)$ its
  Grothendieck group.
  \begin{enumerate}
  \item The element $[0]$ is the zero element in $K_0(\C)$.
  \item If $A$ and $B$ are objects in $\C$ then $[A \oplus B] = [A] +
    [B]$ and $[\Sigma A] = (-1)^n[A]$ in $K_0(\C)$.
  \item When $n$ is odd, then every element in $K_0(\C)$ is of the
    form $[A]$ for some object $A \in \C$.  When $n$ is even, then
    every element in $K_0(\C)$ is of the form $[A]-[B]$ for some
    objects $A,B \in \C$.
  \end{enumerate}
\end{proposition}

\begin{proof}
  (1) If $n$ is even, then $[0]$ is zero in $K_0(\C)$ by definition of
  $R(\C)$.  If $n$ is odd, then the Euler relation corresponding to
  the trivial $n$-angle
  \begin{equation*}
    0 \to 0 \to \cdots \to 0 \to \Sigma 0
  \end{equation*}
  gives that $[0]$ is zero in $K_0(\C)$.
 
  (2) The two $n$-$\Sigma$-sequences
  \begin{align*}
    & A \xrightarrow{1} A \to 0 \to 0 \to \cdots \to 0 \to \Sigma A\\
    & 0 \to B \xrightarrow{1} B \to 0 \to \cdots \to 0 \to 0
  \end{align*}
  are $n$-angles, hence so is their direct sum $S_\bullet$.  From (1),
  the Euler relation $\chi(S_\bullet)$ gives the equality $[A \oplus
  B] = [A] + [B]$ in $K_0(\C)$.  Moreover, the $n$-$\Sigma$-sequence
  \begin{equation*}
    T_\bullet \colon \quad A \to 0 \to 0 \to \cdots \to 0 \to \Sigma A
    \xrightarrow{(-1)^n} \Sigma A
  \end{equation*}
  is an $n$-angle.  From (1), the Euler relation $\chi(T_\bullet)$
  gives the equality $[A] + (-1)^{n+1}[\Sigma A] = 0$ in $K_0(\C)$.
       
  (3) Let $x$ be an element in $K_0(\C)$. If $x=0$ then $x = [0]$ by
  (1), and we are done.  If $x$ is nonzero, then there are
  non-negative integers $a_1, \dots, a_r, b_1, \dots, b_t$ and objects
  $A_1, \dots, A_r, B_1, \dots, B_t$ in $\C$ with
  \begin{equation*}
    x = a_1[A_1] + \cdots + a_r[A_r] - b_1[B_1] - \cdots - b_t[B_t].
  \end{equation*}
  Using (2), we see that
  \begin{equation*}
    x = \left\{
      \begin{array}{ll}
        [A_1^{a_1} \oplus \cdots \oplus A_r^{a_r} \oplus (\Sigma
        B_1)^{b_1} \oplus \cdots \oplus (\Sigma B_t)^{b_t}] &
        \text{when $n$ is odd}\\
        {[ A_1^{a_1} \oplus \cdots \oplus A_r^{a_r}]} - [B_1^{b_1}
        \oplus \cdots \oplus B_t^{b_t}] & \text{when $n$ is even}
      \end{array}
    \right.
  \end{equation*}
  in $K_0( \C )$, where $A^m$ denotes the direct sum of $m$ copies of
  an object $A \in \C$.
\end{proof}

The next result provides an alternative interpretation of the
Grothendieck group in the case when $n$ is odd.

\begin{proposition}
  \label{prop:description}
  Let $n \geq 3$ be an odd integer and $(\C,\Sigma)$ an $n$-angulated
  category.  Consider the following relation on the set of objects of
  $\C$: $A \sim B$ if and only if there exist objects $C_1, \dots,
  C_n$ and two $n$-angles
  \begin{equation*}
    A \oplus C_1 \xrightarrow{\alpha_1} C_2 \xrightarrow{\alpha_2}
    \cdots \xrightarrow{\alpha_{n-1}} C_n \xrightarrow{\alpha_n}
    \Sigma A \oplus \Sigma C_1
  \end{equation*}
  \begin{equation*}
    B \oplus C_1 \xrightarrow{\beta_1} C_2 \xrightarrow{\beta_2}
    \cdots \xrightarrow{\beta_{n-1}} C_n \xrightarrow{\beta_n} \Sigma
    B \oplus \Sigma C_1
  \end{equation*}
  in $\C$. Then the following hold:
  \begin{enumerate}
  \item The relation is an equivalence relation.
  \item The set $\pi$ of equivalence classes $\{A\}$ of objects in
    $\C$ forms an abelian group with addition $\{A\} + \{B\} = \{A
    \oplus B\}$.  The inverse of an element $\{A\}$ is $\{\Sigma A\}$.
  \item The groups $\pi$ and $K_0(\C)$ are isomorphic via the
    correspondence $\{A\} \leftrightarrow [A]$.
  \end{enumerate}
\end{proposition}

\begin{proof}
  (1) The relation is clearly reflexive and symmetric.  Suppose that
  $A \sim B$ and $B \sim D$.  Then by definition, there exist four
  $n$-angles
  \begin{equation*}
    A \oplus C_1 \xrightarrow{\alpha_1} C_2 \xrightarrow{\alpha_2}
    \cdots \xrightarrow{\alpha_{n-1}} C_n \xrightarrow{\alpha_n}
    \Sigma A \oplus \Sigma C_1
  \end{equation*}
  \begin{equation*}
    B \oplus C_1 \xrightarrow{\beta_1} C_2 \xrightarrow{\beta_2}
    \cdots \xrightarrow{\beta_{n-1}} C_n \xrightarrow{\beta_n} \Sigma
    B \oplus \Sigma C_1
  \end{equation*}
  \begin{equation*}
    B \oplus C'_1 \xrightarrow{\gamma_1} C'_2 \xrightarrow{\gamma_2}
    \cdots \xrightarrow{\gamma_{n-1}} C'_n \xrightarrow{\gamma_n}
    \Sigma B \oplus \Sigma C'_1
  \end{equation*}
  \begin{equation*}
    D \oplus C'_1 \xrightarrow{\delta_1} C'_2 \xrightarrow{\delta_2}
    \cdots \xrightarrow{\delta_{n-1}} C'_n \xrightarrow{\delta_n}
    \Sigma D \oplus \Sigma C'_1
  \end{equation*}
  in $\C$.  Now form two new $n$-angles from these: the direct sum of
  the first and the third, and the direct sum of the second and the
  fourth.  These two $n$-angles show that $A \sim D$, hence the
  relation is also transitive.  Consequently, the relation is an
  equivalence relation.

  (2) The operation $\{A\} + \{B\} = \{A \oplus B\}$ in $\pi$ is a
  well-defined associative binary operation, with $\{0\}$ as the
  identity element.  Now for any object $A \in \C$, consider the two
  $n$-angles
  \begin{center}
    \begin{tikzpicture}[text centered]
      \node (X1) at (0,1){$A \oplus \Sigma A$};
      \node (X2) at (1.5,1){$\Sigma A$};
      \node (X3) at (3,1){$\Sigma A$};
      \node (X4) at (4.5,1){$\Sigma A$};
      \node (Xdots) at (6,1){$\cdots$};
      \node (Xn) at (7.5,1){$\Sigma A$};
      \node (SX1) at (9,1){$\Sigma A \oplus \Sigma^2 A$};

      \node (Y1) at (0,0){$0$};
      \node (Y2) at (1.5,0){$\Sigma A$};
      \node (Y3) at (3,0){$\Sigma A$};
      \node (Y4) at (4.5,0){$\Sigma A$};
      \node (Ydots) at (6,0){$\cdots$};
      \node (Yn) at (7.5,0){$\Sigma A$};
      \node (SY1) at (9,0){$0$};

      \begin{scope}[font=\scriptsize,->,midway]
        \draw (X1) -- node[above]{$\left[
            \begin{smallmatrix}
              0 & 1
            \end{smallmatrix}
          \right]$} (X2);
        \draw (X2) -- node[above]{$0$} (X3);
        \draw (X3) -- node[above]{$1$} (X4);
        \draw (X4) -- node[above]{$0$} (Xdots);
        \draw (Xdots) -- node[above]{$0$} (Xn);
        \draw (Xn) -- node[above]{$\left[
            \begin{smallmatrix}
              1 \\ 0
            \end{smallmatrix}
          \right]$} (SX1);

        \draw (Y1) -- node[above]{} (Y2);
        \draw (Y2) -- node[above]{$1$} (Y3);
        \draw (Y3) -- node[above]{$0$} (Y4);
        \draw (Y4) -- node[above]{$1$} (Ydots);
        \draw (Ydots) -- node[above]{$1$} (Yn);
        \draw (Yn) -- node[above]{} (SY1);
      \end{scope}
    \end{tikzpicture}
  \end{center}
  in $\C$, each built from trivial $n$-angles.  They show that $A
  \oplus \Sigma A \sim 0$, so that
  \begin{equation*}
    \{A\} + \{\Sigma A\} = \{A \oplus \Sigma A\} = \{0\}
  \end{equation*}
  in $\pi$.  Hence $\pi$ is an abelian group: the inverse of an
  element $\{A\}$ is $\{\Sigma A\}$.

  (3) We first show that the Euler relations hold in the group $\pi$.
  Suppose that
  \begin{equation*}
    A_{\bullet} \colon \quad A_1 \xrightarrow{\alpha_1} A_2
    \xrightarrow{\alpha_2} \cdots \xrightarrow{\alpha_{n-1}} A_n
    \xrightarrow{\alpha_n} \Sigma A_1
  \end{equation*}
  is an $n$-angle in $\C$.  Take the direct sum of $A_{\bullet}$ and
  the following trivial $n$-angles:
  \begin{center}
    \begin{tikzpicture}[text centered]
      \node (X1) at (0,1){$A_i$};
      \node (X2) at (1,1){$A_i$};
      \node (X3) at (2,1){$0$};
      \node (X4) at (3,1){$0$};
      \node (Xdots) at (4,1){$\cdots$};
      \node (Xn) at (5,1){$0$};
      \node (SX1) at (6,1){$\Sigma A_i$};
      
      \node (Xindex) at (8,1){$i = 3,5, \dots, n$};

      \node (Y1) at (0,0){$0$};
      \node (Y2) at (1,0){$A_i$};
      \node (Y3) at (2,0){$A_i$};
      \node (Y4) at (3,0){$0$};
      \node (Ydots) at (4,0){$\cdots$};
      \node (Yn) at (5,0){$0$};
      \node (SY1) at (6,0){$0$};
      
      \node (Xindex) at (8,0){$i = 4,6, \dots, n-1$};
      
      \node (Z1) at (0,-1){$0$};
      \node (Z2) at (1,-1){$0$};
      \node (Z3) at (2,-1){$A_i$};
      \node (Z4) at (3,-1){$A_i$};
      \node (Zdots) at (4,-1){$\cdots$};
      \node (Zn) at (5,-1){$0$};
      \node (SZ1) at (6,-1){$0$};
      
      \node (Zindex) at (8,-1){$i = 5,7 \dots, n$};
      
      \node (Udots) at (3,-1.4){$\vdots$};
      
      \node (V1) at (0,-2){$0$};
      \node (V2) at (1,-2){$0$};
      \node (Vdots) at (2,-2){$\cdots$};
      \node (Vn-4) at (3,-2){$0$};
      \node (Vn-3) at (4.2,-2){$A_{n-1}$};
      \node (Vn-2) at (5.5,-2){$A_{n-1}$};
      \node (Vn-1) at (6.5,-2){$0$};
      \node (Vn) at (7.5,-2){$0$};
      \node (SV1) at (8.5,-2){$0$};
      
      \node (W1) at (0,-3){$0$};
      \node (W2) at (1,-3){$0$};
      \node (Wdots) at (2,-3){$\cdots$};
      \node (Wn-4) at (3,-3){$0$};
      \node (Wn-3) at (4.2,-3){$0$};
      \node (Wn-2) at (5.5,-3){$A_{n}$};
      \node (Wn-1) at (6.5,-3){$A_n$};
      \node (Wn) at (7.5,-3){$0$};
      \node (SW1) at (8.5,-3){$0$};

      \begin{scope}[font=\scriptsize,->,midway]
        \draw (X1) -- node[above]{$1$} (X2);
        \draw (X2) -- node[above]{} (X3);
        \draw (X3) -- node[above]{} (X4);
        \draw (X4) -- node[above]{} (Xdots);
        \draw (Xdots) -- node[above]{} (Xn);
        \draw (Xn) -- node[above]{} (SX1);
        
        \draw (Y1) -- node[above]{} (Y2);
        \draw (Y2) -- node[above]{$1$} (Y3);
        \draw (Y3) -- node[above]{} (Y4);
        \draw (Y4) -- node[above]{} (Ydots);
        \draw (Ydots) -- node[above]{} (Yn);
        \draw (Yn) -- node[above]{} (SY1);
        
        \draw (Z1) -- node[above]{} (Z2);
        \draw (Z2) -- node[above]{} (Z3);
        \draw (Z3) -- node[above]{$1$} (Z4);
        \draw (Z4) -- node[above]{} (Zdots);
        \draw (Zdots) -- node[above]{} (Zn);
        \draw (Zn) -- node[above]{} (SZ1);
        
        \draw (V1) -- node[above]{} (V2);
        \draw (V2) -- node[above]{} (Vdots);
        \draw (Vdots) -- node[above]{} (Vn-4);
        \draw (Vn-4) -- node[above]{} (Vn-3);
        \draw (Vn-3) -- node[above]{$1$} (Vn-2);
        \draw (Vn-2) -- node[above]{} (Vn-1);
        \draw (Vn-1) -- node[above]{} (Vn);
        \draw (Vn) -- node[above]{} (SV1);
        
        \draw (W1) -- node[above]{} (W2);
        \draw (W2) -- node[above]{} (Wdots);
        \draw (Wdots) -- node[above]{} (Wn-4);
        \draw (Wn-4) -- node[above]{} (Wn-3);
        \draw (Wn-3) -- node[above]{} (Wn-2);
        \draw (Wn-2) -- node[above]{$1$} (Wn-1);
        \draw (Wn-1) -- node[above]{} (Wn);
        \draw (Wn) -- node[above]{} (SW1);
              
      \end{scope}
    \end{tikzpicture}
  \end{center}
  and write the odd integer $n$ as $n=2t+1$.  The result is an
  $n$-angle of the form
  \begin{center}
    \begin{tikzpicture}[text centered]
      \node (X1) at (0,0){$\bigoplus\limits_{i=0}^t A_{2i+1}$};
      \node (X2) at (2,0){$\bigoplus\limits_{i=2}^n A_i$};
      \node (X3) at (4,0){$\bigoplus\limits_{i=3}^n A_i$};
      \node (Xdots) at (6,0){$\cdots$};
      \node (Xn-1) at (8,0){$A_{n-1} \oplus A_n$};
      \node (Xn) at (10,0){$A_n$};
      \node (SX1) at (12,0){$\bigoplus\limits_{i=0}^t \Sigma
        A_{2i+1}$};

      \begin{scope}[font=\scriptsize,->,midway]
        \draw (X1) -- node[above]{} (X2);
        \draw (X2) -- node[above]{} (X3);
        \draw (X3) -- node[above]{} (Xdots);
        \draw (Xdots) -- node[above]{} (Xn-1);
        \draw (Xn-1) -- node[above]{} (Xn);
        \draw (Xn) -- node[above]{} (SX1);

      \end{scope}
    \end{tikzpicture}
  \end{center}
  in $\C$.  Now take the direct sum of the following trivial
  $n$-angles:
  \begin{center}
    \begin{tikzpicture}[text centered]
      \node (X1) at (0,1){$A_i$};
      \node (X2) at (1,1){$A_i$};
      \node (X3) at (2,1){$0$};
      \node (X4) at (3,1){$0$};
      \node (Xdots) at (4,1){$\cdots$};
      \node (Xn) at (5,1){$0$};
      \node (SX1) at (6,1){$\Sigma A_i$};
      
      \node (Xindex) at (8,1){$i = 2,4, \dots, n-1$};

      \node (Y1) at (0,0){$0$};
      \node (Y2) at (1,0){$A_i$};
      \node (Y3) at (2,0){$A_i$};
      \node (Y4) at (3,0){$0$};
      \node (Ydots) at (4,0){$\cdots$};
      \node (Yn) at (5,0){$0$};
      \node (SY1) at (6,0){$0$};
      
      \node (Xindex) at (8,0){$i = 3,5, \dots, n$};
      
      \node (Z1) at (0,-1){$0$};
      \node (Z2) at (1,-1){$0$};
      \node (Z3) at (2,-1){$A_i$};
      \node (Z4) at (3,-1){$A_i$};
      \node (Zdots) at (4,-1){$\cdots$};
      \node (Zn) at (5,-1){$0$};
      \node (SZ1) at (6,-1){$0$};
      
      \node (Zindex) at (8,-1){$i = 4,6 \dots, n-1$};
      
      \node (Udots) at (3,-1.4){$\vdots$};
      
      \node (V1) at (0,-2){$0$};
      \node (V2) at (1,-2){$0$};
      \node (Vdots) at (2,-2){$\cdots$};
      \node (Vn-3) at (3,-2){$0$};
      \node (Vn-2) at (4.2,-2){$A_{n-1}$};
      \node (Vn-1) at (5.5,-2){$A_{n-1}$};
      \node (Vn) at (6.5,-2){$0$};
      \node (SV1) at (7.5,-2){$0$};
      
      \node (W1) at (0,-3){$0$};
      \node (W2) at (1,-3){$0$};
      \node (Wdots) at (2,-3){$\cdots$};
      \node (Wn-3) at (3,-3){$0$};
      \node (Wn-2) at (4.2,-3){$0$};
      \node (Wn-1) at (5.5,-3){$A_{n}$};
      \node (Wn) at (6.5,-3){$A_n$};
      \node (SW1) at (7.5,-3){$0$};

      \begin{scope}[font=\scriptsize,->,midway]
        \draw (X1) -- node[above]{$1$} (X2);
        \draw (X2) -- node[above]{} (X3);
        \draw (X3) -- node[above]{} (X4);
        \draw (X4) -- node[above]{} (Xdots);
        \draw (Xdots) -- node[above]{} (Xn);
        \draw (Xn) -- node[above]{} (SX1);
        
        \draw (Y1) -- node[above]{} (Y2);
        \draw (Y2) -- node[above]{$1$} (Y3);
        \draw (Y3) -- node[above]{} (Y4);
        \draw (Y4) -- node[above]{} (Ydots);
        \draw (Ydots) -- node[above]{} (Yn);
        \draw (Yn) -- node[above]{} (SY1);
        
        \draw (Z1) -- node[above]{} (Z2);
        \draw (Z2) -- node[above]{} (Z3);
        \draw (Z3) -- node[above]{$1$} (Z4);
        \draw (Z4) -- node[above]{} (Zdots);
        \draw (Zdots) -- node[above]{} (Zn);
        \draw (Zn) -- node[above]{} (SZ1);
        
        \draw (V1) -- node[above]{} (V2);
        \draw (V2) -- node[above]{} (Vdots);
        \draw (Vdots) -- node[above]{} (Vn-3);
        \draw (Vn-3) -- node[above]{} (Vn-2);
        \draw (Vn-2) -- node[above]{$1$} (Vn-1);
        \draw (Vn-1) -- node[above]{} (Vn);
        \draw (Vn) -- node[above]{} (SV1);
        
        \draw (W1) -- node[above]{} (W2);
        \draw (W2) -- node[above]{} (Wdots);
        \draw (Wdots) -- node[above]{} (Wn-3);
        \draw (Wn-3) -- node[above]{} (Wn-2);
        \draw (Wn-2) -- node[above]{} (Wn-1);
        \draw (Wn-1) -- node[above]{$1$} (Wn);
        \draw (Wn) -- node[above]{} (SW1);
      \end{scope}
    \end{tikzpicture}
  \end{center}
  The result is an $n$-angle of the form
  \begin{center}
    \begin{tikzpicture}[text centered]
      \node (X1) at (0,0){$\bigoplus\limits_{i=1}^t A_{2i}$};
      \node (X2) at (2,0){$\bigoplus\limits_{i=2}^n A_i$};
      \node (X3) at (4,0){$\bigoplus\limits_{i=3}^n A_i$};
      \node (Xdots) at (6,0){$\cdots$};
      \node (Xn-1) at (8,0){$A_{n-1} \oplus A_n$};
      \node (Xn) at (10,0){$A_n$};
      \node (SX1) at (12,0){$\bigoplus\limits_{i=1}^t \Sigma A_{2i}$};

      \begin{scope}[font=\scriptsize,->,midway]
        \draw (X1) -- node[above]{} (X2);
        \draw (X2) -- node[above]{} (X3);
        \draw (X3) -- node[above]{} (Xdots);
        \draw (Xdots) -- node[above]{} (Xn-1);
        \draw (Xn-1) -- node[above]{} (Xn);
        \draw (Xn) -- node[above]{} (SX1);

      \end{scope}
    \end{tikzpicture}
  \end{center}
  in $\C$.  By definition of the equivalence relation, the two
  $n$-angles show that
  \begin{equation*}
    A_1 \oplus A_3 \oplus \cdots \oplus A_n \sim A_2 \oplus A_4 \oplus
    \cdots \oplus A_{n-1},
  \end{equation*}
  so that the relation
  \begin{equation*}
    \{A_1\} + \{A_3\} + \cdots + \{A_n\} = \{A_2\} + \{A_4\} + \cdots
    + \{A_{n-1}\}
  \end{equation*}
  holds in the group $\pi$.  But this is precisely the Euler relation
  $\Sigma_{i=1}^n (-1)^{i+1} \{A_i\} = 0$ corresponding to the
  $n$-angle $A_{\bullet}$ we started with.

  Since the Euler relations hold in $\pi$, the map $f \colon K_0 (\C)
  \to \pi$ given by $[A] \mapsto \{A\}$ is a well-defined surjective
  homomorphism of abelian groups.  To show injectivity, suppose that
  $f([A]) = 0$.  Then $A \sim 0$, and so there exist objects $C_1,
  \dots, C_n$ and two $n$-angles
  \begin{equation*}
    A \oplus C_1 \xrightarrow{\alpha_1} C_2 \xrightarrow{\alpha_2}
    \cdots \xrightarrow{\alpha_{n-1}} C_n \xrightarrow{\alpha_n}
    \Sigma A \oplus \Sigma C_1
  \end{equation*}
  \begin{equation*}
    C_1 \xrightarrow{\beta_1} C_2 \xrightarrow{\beta_2} \cdots
    \xrightarrow{\beta_{n-1}} C_n \xrightarrow{\beta_n} \Sigma C_1
  \end{equation*}
  in $\C$.  Combining the Euler relations corresponding to these two
  $n$-angles gives $[A] = 0$ in $K_0(\C)$, hence the map $f$ is an
  isomorphism.
\end{proof}

We end this section with the following corollary to Proposition
\ref{prop:description}; it follows immediately from the definition of
the isomorphism between the Grothendieck group and the group of
equivalence classes.  It provides a criterion for determining when two
elements in the Grothendieck group are equal.  For $n = 3$ this is
\cite[Lemma 2.4]{Thomason}, which in turn was based on \cite[1.6,
Corollary]{Landsburg}.

\begin{corollary}\label{cor:equality}
  Let $n \geq 3$ be an odd integer and $(\C,\Sigma)$ an $n$-angulated
  category.  Then for any objects $A,B \in \C$ the following are
  equivalent:
  \begin{enumerate}
  \item $[A] = [B]$ in $K_0(\C)$.
  \item There exist objects $C_1, \dots, C_n$ and two $n$-angles
    \begin{equation*}
      A \oplus C_1 \xrightarrow{\alpha_1} C_2 \xrightarrow{\alpha_2}
      \cdots \xrightarrow{\alpha_{n-1}} C_n \xrightarrow{\alpha_n}
      \Sigma A \oplus \Sigma C_1
    \end{equation*}
    \begin{equation*}
      B \oplus C_1 \xrightarrow{\beta_1} C_2 \xrightarrow{\beta_2}
      \cdots \xrightarrow{\beta_{n-1}} C_n \xrightarrow{\beta_n}
      \Sigma B \oplus \Sigma C_1
    \end{equation*}
    in $\C$.
  \end{enumerate}
\end{corollary}

\section{Example: cluster tilting categories}
\label{sec:cluster}
The first class of examples of $n$-angulated categories appeared in
\cite{GeissKellerOppermann}, namely, those that arise from certain
cluster tilting subcategories of triangulated categories.  In this
section, we show that there always exists a surjective homomorphism
from the Grothendieck group of such an $n$-angulated category onto the
Grothendieck group of the underlying triangulated category.

We recall the construction of these $n$-angulated categories.  Let
$\T$ be a triangulated category with suspension $\Sigma$, and $\C$ a
full subcategory.  Recall that a morphism $C \xrightarrow{f} X$ in
$\T$ is a \emph{right $\C$-approximation} of the object $X$ if the
following hold: the object $C$ belongs to $\C$, and for every morphism
$C' \xrightarrow{g} X$ with $C' \in \C$ there exists a morphism $C'
\xrightarrow{h} C$ such that $g = f \circ h$.
\begin{center}
  \begin{tikzpicture}[text centered]
    \node (C) at (-2,0){$C$};
    \node (X) at (0,0){$X$};
    \node (C') at (0,1.5){$C'$};
           
    \begin{scope}[->,font=\scriptsize,midway]
            
      \draw (C) -- node[above]{$f$} (X);
      \draw (C') -- node[right]{$g$} (X);
      \draw[dashed] (C') -- node[above left]{$h$} (C);
    \end{scope}
  \end{tikzpicture}
\end{center}
A \emph{left $\C$-approximation} of $X$ is defined dually.  The
subcategory $\C$ is called \emph{contravariantly finite} in $\T$ if
every object in $\T$ admits a right $\C$-approximation, and
\emph{covariantly finite} in $\T$ if every object in $\T$ admits a
left $\C$-approximation.  If $\C$ is both contravariantly and
covariantly finite in $\T$, then it is called \emph{functorially
  finite}.  Finally, if $t \ge 2$ is an integer, then the subcategory
$\C$ is called a \emph{$t$-cluster tilting} subcategory of $\T$ if the
following hold:
\begin{enumerate}
\item $\C$ is functorially finite in $\T$.
\item $\C$ is the full subcategory of $\T$ given by
  \begin{align*}
    \C & = \{A \in \T \mid \Hom_{\T}(A, \Sigma^iC) = 0 \text{ for } 1
    \leq i \leq t - 1 \text{ and for all } C \in \C\}\\
    & = \{B \in \T \mid \Hom_{\T}(\Sigma^iC,B) = 0 \text{ for } 1 \leq
    i \leq t - 1 \text{ and for all } C \in \C\}.
\end{align*}
\end{enumerate}

Suppose now that the subcategory $\C$ is $(n - 2)$-cluster tilting for
some integer $n$, and that $\C$ is closed under the automorphism
$\Sigma^{n-2}$ of $\T$; we denote $\Sigma^{n-2}$ by $\susp$.  Let
$\nang$ be the collection of $n$-$\susp$-sequences in $\C$ such that
there exists a diagram
\begin{center}
  \begin{tikzpicture}[text centered]
    \node (A2) at (1.5,1.5){$A_2$};
    \node (A3) at (4.5,1.5){$A_3$};
    \node (Adots) at (5.5,1.5){$\cdots$};
    \node (An-2) at (6.5,1.5){$A_{n-2}$};
    \node (An-1) at (9.5,1.5){$A_{n-1}$};
    
    \node (A1) at (0,0){$A_1$};
    \node (X1) at (3,0){$X_1$};
    \node (X2) at (6,0){$X_2$};
    \node (Xdots) at (7,0){$\cdots$};
    \node (Xn-3) at (8,0){$X_{n-3}$};
    \node (An) at (11,0){$A_n$};
    
    \node (T1) at (1.5,0.8){$\Delta$};
    \node (T2) at (4.5,0.8){$\Delta$};
    \node (Tn-2) at (9.5,0.8){$\Delta$};
    
    \begin{scope}[->,font=\scriptsize,midway]
      \draw (A1) -- node[above left]{$\alpha_1$} (A2);
      \draw (A2) -- node[above right]{$f_1$} (X1);
      \draw (X1) -- node[above left]{$g_1$} (A3);
      \draw (A3) -- node[above right]{$f_2$} (X2);
      \draw (An-2) -- node[above right]{$f_{n-3}$} (Xn-3);
      \draw (Xn-3) -- node[above left]{$g_{n-3}$} (An-1);
      \draw (An-1) -- node[above right]{$\alpha_{n-1}$} (An);
      
      \draw (A2) -- node[above]{$\alpha_2$} (A3);
      \draw (An-2) -- node[above]{$\alpha_{n-2}$} (An-1);
      \draw[|->] (X1) -- node[above]{$\partial_{n-2}$} (A1);
      \draw[|->] (X2) -- node[above]{$\partial_{n-3}$} (X1);
      \draw[|->] (An) -- node[above]{$\partial_1$} (Xn-3);     
    \end{scope}
  \end{tikzpicture}
\end{center}
in $\T$ with the following properties:
\begin{enumerate}
\item Each diagram triangle $\Delta$ is a triangle in $\T$, where a
  map $X \mapsto Y$ denotes a map from $X$ to $\Sigma Y$.
\item The other diagram triangles commute.
\item The map $\alpha_n$ equals the composition
  $\Sigma^{n-3} \partial_{n-2} \circ \Sigma^{n-4} \partial_{n-3} \circ
  \cdots \circ \partial_1$.
\end{enumerate}
Then it is shown in \cite[Section 3, Theorem 1]{GeissKellerOppermann}
that $(\C,\susp)$ is an $n$-angulated category, with the collection
$\nang$ as $n$-angles.

The following result shows that, in the above situation, the
``obvious'' map from from the Grothendieck group of $\C$ to that of
$\T$ is a surjective homomorphism of groups.

\begin{theorem}
  \label{thm:cluster}
  Let $(\T,\Sigma)$ be a triangulated category with an
  $(n-2)$-cluster-tilting subcategory $\C$ which is closed under
  $\Sigma^{n-2}$.  Furthermore, let $(\C, \susp)$ be the corresponding
  $n$-angulated category, where $\susp = \Sigma^{n-2}$. Then the map
  \begin{equation*}
    \begin{array}{r@{\ }c@{\ }l}
      K_0(\C) & \to & K_0(\T)\\
      {[A]-[B]} & \mapsto & [A]-[B]
    \end{array}
  \end{equation*}
  is a well-defined surjective homomorphism of groups. 
\end{theorem}

\begin{remark}
  (1) When $n$ is an odd integer, then we know from Proposition
  \ref{prop:elementary-properties} that every element in $K_0 ( \C )$
  is just of the form $[A]$ for some object $A \in \C$.  It is to be
  understood that in this case, the map sends $[A]$ in $K_0(\C)$ to
  $[A]$ in $K_0(\T)$.

  (2) Using Corollary \ref{cor:equality}, it is possible to
  characterize the kernel of the surjective homomorphism $K_0(\C) \to
  K_0(\T)$.  However, this characterization is not very effective,
  since it uses triangles in $\T$.
\end{remark}

\begin{proof}
  Let $F( \C )$ be the free abelian group on the set of isomorphism
  classes $\langle A \rangle$ of objects $A \in \C$, and $f \colon F (
  \C ) \to K_0( \T )$ the group homomorphism given by
  \begin{equation*}
    a_1 \langle A_1 \rangle + \cdots + a_s \langle A_s \rangle - b_1
    \langle B_1 \rangle - \cdots - b_t \langle B_t \rangle \mapsto a_1
    [A_1] + \cdots + a_s [A_s] - b_1 [B_1] - \cdots - b_t [B_t].
  \end{equation*}
  In order to show that $f$ extends to a homomorphism $\hat{f} \colon
  K_0(\C) \to K_0(\T)$, we must prove that $f$ maps the subgroup
  $R(\C)$ in Definition \ref{def:K0} to zero.  In other words, we must
  show that for every $n$-angle in $(\C,\susp)$, the corresponding
  Euler relation holds in $K_0(\T)$.

  Let therefore 
  \begin{equation*}
    A_1 \xrightarrow{\alpha_1} A_2 \xrightarrow{\alpha_2} \cdots
    \xrightarrow{\alpha_{n-1}} A_n \xrightarrow{\alpha_n} \susp A_1
  \end{equation*}
  be an $n$-angle in $(\C,\susp)$.  By definition, this $n$-angle is
  built from triangles in $(\T,\Sigma)$, that is, there exists a
  diagram
  \begin{center}
    \begin{tikzpicture}[text centered]
      \node (A2) at (1.5,1.5){$A_2$};
      \node (A3) at (4.5,1.5){$A_3$};
      \node (Adots) at (5.5,1.5){$\cdots$};
      \node (An-2) at (6.5,1.5){$A_{n-2}$};
      \node (An-1) at (9.5,1.5){$A_{n-1}$};
      
      \node (A1) at (0,0){$A_1$};
      \node (X1) at (3,0){$X_1$};
      \node (X2) at (6,0){$X_2$};
      \node (Xdots) at (7,0){$\cdots$};
      \node (Xn-3) at (8,0){$X_{n-3}$};
      \node (An) at (11,0){$A_n$};
      
      \node (T1) at (1.5,0.8){$\Delta$};
      \node (T2) at (4.5,0.8){$\Delta$};
      \node (Tn-2) at (9.5,0.8){$\Delta$};
      
      \begin{scope}[->,font=\scriptsize,midway]
        \draw (A1) -- node[above left]{$\alpha_1$} (A2);
        \draw (A2) -- node[above right]{$f_1$} (X1);
        \draw (X1) -- node[above left]{$g_1$} (A3);
        \draw (A3) -- node[above right]{$f_2$} (X2);
        \draw (An-2) -- node[above right]{$f_{n-3}$} (Xn-3);
        \draw (Xn-3) -- node[above left]{$g_{n-3}$} (An-1);
        \draw (An-1) -- node[above right]{$\alpha_{n-1}$} (An);
        
        \draw (A2) -- node[above]{$\alpha_2$} (A3);
        \draw (An-2) -- node[above]{$\alpha_{n-2}$} (An-1);
        \draw[|->] (X1) -- node[above]{$\partial_{n-2}$} (A1);
        \draw[|->] (X2) -- node[above]{$\partial_{n-3}$} (X1);
        \draw[|->] (An) -- node[above]{$\partial_1$} (Xn-3);     
      \end{scope}
    \end{tikzpicture}
  \end{center}
  in $\T$ with the three properties given prior to the theorem.  In
  particular, the $n - 2$ diagram triangles $\Delta$ are triangles in
  $(\T, \Sigma )$.  The Euler relations in $K_0 ( \T )$ therefore give
  \begin{eqnarray*}
    0 & = & [A_1] - [A_2] + [X_1]\\
    & = & [A_1] - [A_2] + [A_3] - [X_2]\\
    & \vdots &\\
    & = & [A_1] - [A_2] + [A_3] - \cdots + (-1)^{n+1} [A_n],
  \end{eqnarray*}
  so that $f$ maps the subgroup $R(\C)$ of $F(\C)$ to zero in
  $K_0(\T)$.  Consequently, the homomorphism $f$ extends to a
  homomorphism $\hat{f} \colon K_0(\C) \to K_0(\T)$, and the latter is
  precisely the homomorphism given in the theorem.

  It remains to show that $\hat{f}$ is surjective. Let therefore $[X]$
  be an element in $K_0(\T)$. By \cite[5.5,Proposition]{KellerReiten},
  since $\C$ is an $(n - 2)$-cluster-tilting subcategory of $\T$,
  every object in $\T$ admits a $\C$-resolution of length
  $n-3$.  Consequently, there exist $n-3$ triangles
  \begin{center}
    \begin{tikzpicture}[text centered]
      \node (X1) at (0,2){$X_1$};
      \node (X2) at (1.5,2){$A_0$};
      \node (X3) at (3,2){$X$};
      \node (SX1) at (4.5,2){$\Sigma X_1$};
      
      \node (Y1) at (0,1){$X_2$};
      \node (Y2) at (1.5,1){$A_1$};
      \node (Y3) at (3,1){$X_1$};
      \node (SY1) at (4.5,1){$\Sigma X_2$};
      
      \node (D1) at (0,0.5){$\vdots$};
      \node (D2) at (1.5,0.5){$\vdots$};
      \node (D3) at (3,0.5){$\vdots$};
      \node (SD1) at (4.5,0.5){$\vdots$};
      
      \node (Z1) at (0,-0.2){$X_{n-4}$};
      \node (Z2) at (1.5,-0.2){$A_{n-5}$};
      \node (Z3) at (3,-0.2){$X_{n-5}$};
      \node (SZ1) at (4.5,-0.2){$\Sigma X_{n-4}$};
      
      \node (W1) at (0,-1.2){$A_{n-3}$};
      \node (W2) at (1.5,-1.2){$A_{n-4}$};
      \node (W3) at (3,-1.2){$X_{n-4}$};
      \node (SW1) at (4.5,-1.2){$\Sigma A_{n-3}$};
      
      \begin{scope}[font=\scriptsize,->,midway]      
        \draw (X1) -- node[above]{} (X2);
        \draw (X2) -- node[above]{} (X3);
        \draw (X3) -- node[above]{} (SX1);
        
        \draw (Y1) -- node[above]{} (Y2);
        \draw (Y2) -- node[above]{} (Y3);
        \draw (Y3) -- node[above]{} (SY1);
        
        \draw (Z1) -- node[above]{} (Z2);
        \draw (Z2) -- node[above]{} (Z3);
        \draw (Z3) -- node[above]{} (SZ1);
        
        \draw (W1) -- node[above]{} (W2);
        \draw (W2) -- node[above]{} (W3);
        \draw (W3) -- node[above]{} (SW1);
      \end{scope}
    \end{tikzpicture}
  \end{center}
  in $\T$, with $A_i \in \C$ for all $i$. The Euler relations in
  $K_0(\C)$ now gives
  \begin{eqnarray*}
    [X] & = & [A_0] - [X_1]\\
    & = & [A_0] - [A_1] + [X_2]\\
    & \vdots &\\
    & = & [A_0] - [A_1] + [A_2] - \cdots + (-1)^{n-3} [A_{n-3}],
  \end{eqnarray*}
  hence $[X] = \hat{f} \left(\Sigma_{i=0}^{n-3} (-1)^i [A_i]\right)$.
  This shows that the homomorphism $\hat{f}$ is surjective.
\end{proof}

\section{Classifying subcategories}
\label{sec:classifying} 
In this section we prove a generalized version of Thomason's
classification theorem \cite[Theorem 2.1]{Thomason}.  Thomason's
theorem states that the subgroups of the Grothendieck group of a
triangulated category correspond bijectively to the so-called dense
triangulated subcategories.  We generalize this to $n$-angulated
categories, for $n$ odd.  The reason why our proof does not work when
$n$ is even is very simple: in this case the relation $[\Sigma C] = -
[C]$ does not hold in the Grothendieck group (as it does when $n$ is
odd), but $[\Sigma C] = [C]$ instead.

We start by defining $n$-angulated subcategories of an $n$-angulated
category $(\C,\Sigma)$.  In general, if $(\C',\Sigma')$ is another
$n$-angulated category, then a functor $L \colon \C \to \C'$ is
\emph{$n$-angulated} if it has the following three properties:
\begin{enumerate}
\item $L$ is additive.
\item There exists a natural isomorphism $\eta \colon L \circ \Sigma
  \to \Sigma' \circ L$.
\item $L$ preserves $n$-angles: if 
  \begin{equation*}
    A_1 \xrightarrow{\alpha_1} A_2 \xrightarrow{\alpha_2} \cdots
    \xrightarrow{\alpha_{n - 1}} A_n \xrightarrow{\alpha_n} \Sigma A_1
  \end{equation*}
  is an $n$-angle in $\C$, then
  \begin{equation*}
    L(A_1) \xrightarrow{L( \alpha_1)} L(A_2) \xrightarrow{L(
      \alpha_2)} \cdots \xrightarrow{L( \alpha_{n - 1})} L(A_n)
    \xrightarrow{\eta \circ L(\alpha_n)} \Sigma' L(A_1) 
  \end{equation*}
  is an $n$-angle in $\C'$.
\end{enumerate}
As in the triangulated case, the key requirement of an $n$-angulated
subcategory is that the inclusion functor is $n$-angulated.

\begin{definition}
  \label{def:fullsub}
  Let $(\C,\Sigma)$ be an $n$-angulated category.  
  \begin{enumerate}
  \item An \emph{$n$-angulated subcategory} of $\C$ is a full
    subcategory $\A$ such that $(\A,\Sigma)$ is $n$-angulated, closed
    under isomorphisms and the inclusion functor $\iota \colon \A \to
    \C$ is an $n$-angulated functor.
  \item A subcategory $\A$ of $\C$ is \emph{dense} in $\C$ if the
    following holds: every object in $\C$ is a direct summand of an
    object in $\A$.
   \item A subcategory $\A$ of $\C$ is \emph{complete} if the
     following holds: given an $n$-angle in $\C$ in which $n-1$ of the
     vertices are in $\A$, then the last vertex is an object in $\A$.
  \end{enumerate}
\end{definition}

Note that in the triangulated case, i.e.\ when $n = 3$, then every
triangulated subcategory of a triangulated category is complete.  This
is the well-known ``two out of three'' property for triangulated
subcategories, and this was how triangulated subcategories were
originally defined (cf.\ \cite[1.1]{Thomason} and \cite[\S 1, no.\ 1,
2-3]{Verdier}).  The reason is that, up to isomorphism, there is only
one way to complete a given map to a triangle (using axiom (N1)(c)), a
consequence of the Five Lemma.  Thus, when $\C$ is a triangulated
category and $\A$ is a triangulated subcategory, then the triangles in
$\A$ are characterized as follows: a $3$-$\Sigma$-sequence in $\A$ (or
equivalently, a $3$-$\Sigma$-sequence in $\C$ with all its objects in
$\A$) is a triangle in $\A$ if and only if it is a triangle in $\C$.
As the following lemma shows, this holds for general $n$-angulated
categories.

\begin{lemma}
  \label{lem:sequences}
  Let $(\C,\Sigma)$ be an $n$-angulated category, and $(\A,\Sigma)$ an
  $n$-angulated subcategory.  Then an $n$-$\Sigma$-sequence in $\A$ is
  an $n$-angle in $\A$ if and only if it is an $n$-angle in $\C$.
\end{lemma}

\begin{proof}
  Consider an $n$-$\Sigma$-sequence 
  \begin{equation*}
    A_{\bullet} \colon \quad A_1 \xrightarrow{\alpha_1} A_2
    \xrightarrow{\alpha_2} \cdots \xrightarrow{\alpha_{n - 1}} A_n
    \xrightarrow{\alpha_n} \Sigma A_1 
  \end{equation*}
  in $\A$.  If $A_{\bullet}$ is an $n$-angle in $\A$, then it is also
  an $n$-angle in $\C$, since the inclusion functor $\iota \colon \A
  \to \C$ is $n$-angulated.  Conversely, suppose that $A_{\bullet}$ is
  an $n$-angle in $\C$.  Using axiom (N1)(c), we can complete the map
  $\alpha_1 \colon A_1 \to A_2$ to an $n$-angle
  \begin{equation*}
    B_{\bullet} \colon \quad A_1 \xrightarrow{\alpha_1} A_2
    \xrightarrow{\beta_2} B_3 \xrightarrow{\beta_3} \cdots
    \xrightarrow{\beta_{n - 1}} B_n \xrightarrow{\beta_n} \Sigma A_1 
  \end{equation*}
  in $\A$.  As above, this $n$-angle is also an $n$-angle in $\C$. Now
  use axiom (N3) to obtain a morphism
  \begin{center}
    \begin{tikzpicture}[text centered]
      \node (X0) at (-1.5,1.5){$B_{\bullet} \colon$};
      \node (X1) at (0,1.5){$A_1$};
      \node (X2) at (1.5,1.5){$A_2$};
      \node (X3) at (3,1.5){$B_3$};
      \node (Xdots) at (4.5,1.5){$\cdots$};
      \node (Xn) at (6,1.5){$B_n$};
      \node (SX1) at (7.5,1.5){$\Sigma A_1$};

      \node (Y0) at (-1.5,0){$A_{\bullet} \colon$};
      \node (Y1) at (0,0){$A_1$};
      \node (Y2) at (1.5,0){$A_2$};
      \node (Y3) at (3,0){$A_3$};
      \node (Ydots) at (4.5,0){$\cdots$};
      \node (Yn) at (6,0){$A_n$};
      \node (SY1) at (7.5,0){$\Sigma A_1$};

      \begin{scope}[font=\scriptsize,->,midway]
        \draw (X0) -- node[right]{$\varphi$} (Y0);
        \draw[-,double equal sign distance] (X1) -- (Y1);
        \draw[-,double equal sign distance] (X2) -- (Y2);
        \draw[dashed] (X3) -- node[right]{$\varphi_3$} (Y3);
        \draw[dashed] (Xn) -- node[right]{$\varphi_n$} (Yn);
        \draw[-,double equal sign distance] (SX1) -- (SY1);
      
        \draw (X1) -- node[above]{$\alpha_1$} (X2);
        \draw (X2) -- node[above]{$\beta_2$} (X3);
        \draw (X3) -- node[above]{$\beta_3$} (Xdots);
        \draw (Xdots) -- node[above]{$\beta_{n - 1}$} (Xn);
        \draw (Xn) -- node[above]{$\beta_n$} (SX1);
        \draw (Y1) -- node[above]{$\alpha_1$} (Y2);
        \draw (Y2) -- node[above]{$\alpha_2$} (Y3);
        \draw (Y3) -- node[above]{$\alpha_3$} (Ydots);
        \draw (Ydots) -- node[above]{$\alpha_{n - 1}$} (Yn);
        \draw (Yn) -- node[above]{$\alpha_n$} (SY1);
      \end{scope}
    \end{tikzpicture}
  \end{center}
  of $n$-angles in $\C$.  Since $\A$ is a full subcategory of $\C$,
  the maps $\varphi_3, \dots, \varphi_n$ belong to $\A$.  For the same
  reason, the $n$-$\Sigma$-sequence $A_{\bullet}$ is exact in $\A$,
  since it is exact in $\C$ by \cite[Proposition
  1.5]{GeissKellerOppermann}.  Consequently, the morphism $\varphi
  \colon B_{\bullet} \to A_{\bullet}$ is a weak isomorphism from an
  $n$-angle in $\A$ to an exact $n$-$\Sigma$-sequence in $\A$.  By
  \cite[Lemma 1.4]{GeissKellerOppermann}, the $n$-$\Sigma$-sequence
  $A_{\bullet}$ is an $n$-angle in $\A$.
\end{proof}

Our proof of the classification theorem is modeled on that provided by
Thomason.  We start with the following technical lemma.

\begin{lemma}
  \label{lem:technical}
  Let $(\C,\Sigma)$ be an $n$-angulated category, $(\A,\Sigma)$ a
  complete $n$-angulated subcategory, and $C$ an object in $\C$.  If
  there exists and object $A \in \A$ such that $C \oplus A \in \A$,
  then $C \in \A$.
\end{lemma}

\begin{proof}
  The $n$-$\Sigma$-sequence
  \begin{equation*}
    C \xrightarrow{\left[
        \begin{smallmatrix}
          1\\
          0
        \end{smallmatrix}
      \right]} C\oplus A \xrightarrow{\left[
        \begin{smallmatrix}
          0 & 1
        \end{smallmatrix}
      \right]} A \to 0 \to \cdots \to 0 \to \Sigma C
  \end{equation*}
  is an $n$-angle in $\C$, since it is the direct sum of trivial
  $n$-angles.  The $n - 1$ last vertices are all objects in $\A$,
  hence, since $\A$ is complete, so is $C$.
\end{proof}

From now on we need to restrict to the case when $n$ is odd, as
explained in the beginning of this section.  The next result is an
$n$-angulated version of \cite[Lemma 2.2]{Thomason}.  Note first that
if $(\C,\Sigma)$ is an $n$-angulated category and $(\A,\Sigma)$ an
$n$-angulated subcategory, then the inclusion functor $\iota \colon \A
\to \C$ induces a homomorphism
\begin{equation*}
  K_0(\A) \to K_0(\C)
\end{equation*}
of Grothendieck groups, since the functor maps $n$-angles in $\A$ to
$n$-angles in $\C$.  We shall denote the image of this homomorphism by
$\Image K_0(\A)$; this is a subgroup of $K_0(\C)$.

\begin{lemma}
  \label{lem:subcategory}
  Let $n \geq 3$ be an odd integer, let $(\C,\Sigma)$ be an
  $n$-angulated category and $(\A,\Sigma)$ a complete and dense
  $n$-angulated subcategory of $(\C,\Sigma)$.  Then for any object
  $C\in \C$ the following holds: $C$ belongs to $\A$ if and only if
  $[C] = 0$ in $K_0(\C)/\Image K_0(\A)$.
\end{lemma}

\begin{proof}
  If $C$ belongs to $\A$, then in $K_0(\C)$ the element $[C]$
  obviously belongs to the subgroup $\Image K_0(\A)$.  Thus $[C] = 0$
  in $K_0(\C)/\Image K_0(\A)$.

  Conversely, suppose that $[C] = 0$ in $K_0(\C)/\Image K_0(\A)$.
  Then, as above, in $K_0(\C)$ the element $[C]$ belongs to the
  subgroup $\Image K_0(\A)$, hence there is an object $A \in \A$ such
  that $[C] = [A]$ in $K_0(\C)$.  By Corollary \ref{cor:equality},
  there exist objects $C_1, \dots, C_n$ and two $n$-angles
  \begin{equation*}
    C \oplus C_1 \xrightarrow{\alpha_1} C_2 \xrightarrow{\alpha_2}
    \cdots \xrightarrow{\alpha_{n-1}} C_n \xrightarrow{\alpha_n} \Sigma
    C \oplus \Sigma C_1
  \end{equation*}
  \begin{equation*}
    A \oplus C_1 \xrightarrow{\beta_1} C_2 \xrightarrow{\beta_2} \cdots
    \xrightarrow{\beta_{n-1}} C_n \xrightarrow{\beta_n} \Sigma A \oplus
    \Sigma C_1
  \end{equation*}
  in $\C$.  For each $i$, choose an object $C_i'$ such that $C_i
  \oplus C_i'$ belongs to $\A$; there exists such an object since $\A$
  is dense.  Furthermore, define the object $\overline{C}$ by
  \begin{equation*}
    \overline{C} = C_1 \oplus C_2' \oplus C_3 \oplus \cdots \oplus
    C_{n-1}' \oplus C_n.
  \end{equation*}
  By adding trivial $n$-angles involving the objects $C_i$ and $C_i'$
  to the two $n$-angles above, we obtain two new $n$-angles
  \begin{center}
    \begin{tikzpicture}[text centered]
      \node (X1) at (0,0){$C \oplus \overline{C}$};
      \node (X2) at (2.5,0){$\bigoplus\limits_{i=2}^n (C_i \oplus
        C_i')$};
      \node (X3) at (5.5,0){$\bigoplus\limits_{i=3}^n (C_i \oplus
        C_i')$};
      \node (Xdots) at (7.8,0){$\cdots$};
      \node (Xn) at (9.8,0){$C_n \oplus C_n'$};
      \node (SX1) at (12,0){$\Sigma C \oplus \Sigma \overline{C}$};
      
      \node (Y1) at (0,-1){$A \oplus \overline{C}$};
      \node (Y2) at (2.5,-1){$\bigoplus\limits_{i=2}^n (C_i \oplus
        C_i')$};
      \node (Y3) at (5.5,-1){$\bigoplus\limits_{i=3}^n (C_i \oplus
        C_i')$};
      \node (Ydots) at (7.8,-1){$\cdots$};
      \node (Yn) at (9.8,-1){$C_n \oplus C_n'$};
      \node (SY1) at (12,-1){$\Sigma A \oplus \Sigma \overline{C}$};

      \begin{scope}[font=\scriptsize,->,midway]
        \draw (X1) -- node[above]{} (X2);
        \draw (X2) -- node[above]{} (X3);
        \draw (X3) -- node[above]{} (Xdots);
        \draw (Xdots) -- node[above]{} (Xn);
        \draw (Xn) -- node[above]{} (SX1);
        
        \draw (Y1) -- node[above]{} (Y2);
        \draw (Y2) -- node[above]{} (Y3);
        \draw (Y3) -- node[above]{} (Ydots);
        \draw (Ydots) -- node[above]{} (Yn);
        \draw (Yn) -- node[above]{} (SY1);

      \end{scope}
    \end{tikzpicture}
  \end{center}
  in $\C$.  In both of these, the last $n-1$ vertices are objects in
  $\A$, hence, since $\A$ is dense, so are the objects $C \oplus
  \overline{C}$ and $A \oplus \overline{C}$.  Now, all the three
  objects $A$, $A \oplus \overline{C}$ and $C \oplus \overline{C}$
  belong to $\A$.  Using Lemma \ref{lem:technical} twice, we see that
  $\overline{C}$ and $C$ also belong to $\A$.
\end{proof}

We need one more lemma before we can prove the classification theorem.

\begin{lemma}
  \label{lem:subgroup}
  Let $n \geq 3$ be an odd integer and $(\C,\Sigma)$ an $n$-angulated
  category.  For a subgroup $H$ of $K_0(\C)$, denote by $\A_H$ the
  full subcategory of $\C$ consisting of those objects $A$ in $\C$
  such that $[A] \in H \leq K_0(\C)$.  Then $(\A_H,\Sigma)$ is a
  complete and dense $n$-angulated subcategory of $(\C,\Sigma)$, by
  declaring the $n$-angles in $\C$ with all objects in $\A_H$ to be
  the $n$-angles in $\A_H$.
\end{lemma}

\begin{proof}
  We first show that $\A_H$ is dense and complete.  If $C$ is any
  object in $\C$, then
  \begin{equation*}
    [C \oplus \Sigma C] = [C] + [\Sigma C] = [C] - [C] = 0 \in H
  \end{equation*}
  in $K_0(\C)$, hence $C \oplus \Sigma C$ belongs to $\A_H$ by
  definition.  This shows that $\A_H$ is dense in $\C$.  To prove
  completeness, suppose that
  \begin{equation*}
    C_1 \xrightarrow{} C_2 \xrightarrow{} \cdots \xrightarrow{} C_n
    \xrightarrow{} \Sigma C_1
  \end{equation*}
  is an $n$-angle in $\C$ in which all but possibly one of the
  vertices, say $C_t$, are objects in $\A_H$.  The Euler relation in
  $K_0(\C)$ gives $\Sigma_{i=1}^n (-1)^{i+1} [C_i] =0 \in H$, and
  since all the $[C_i]$ with $i \neq t$ are elements in $H$, we obtain
  $[C_t] \in H$.  But then $C_t \in \A_H$ by definition, hence $\A_H$
  is complete.

  To show that $(\A_H,\Sigma)$ is an $n$-angulated category, we must
  prove that $\A_H$ is closed under the automorphism $\Sigma$, and
  verify that the collection of declared $n$-angles in $\A_H$
  satisfies the four axioms (N1)--(N4).

  The first part is easy.  For if $A \in \A_H$, then by definition
  $[A] \in H$ in $K_0(\C)$, and then $-[A] \in H$ since $H$ is a
  subgroup of $K_0(\C)$.  But $[\Sigma A] = -[A] = [\Sigma^{-1} A]$,
  so both $\Sigma A$ and $\Sigma^{-1} A$ belong to $\A_H$.  Hence
  $\A_H$ is closed under the automorphism $\Sigma$.

  Next, we verify that the collection $\nang_{\A_H}$ of declared
  $n$-angles in $\A_H$ satisfies the four $n$-angle axioms.  Note
  first that $\A_H$ is closed under isomorphisms.  For if $A \in \A_H$
  and $C \in \C$ are isomorphic objects, then $[C] = [A] \in H$ in
  $K_0(\C)$, so that $C \in \A_H$.  Now combine this fact with the
  fact that $\A_H$ is a full subcategory of $\C$, and that the
  collection $\nang_{\A_H}$ consists precisely of those $n$-angles in
  $\C$ with all objects in $\A_H$.  It follows that the collection
  $\nang_{\A_H}$ satisfies all the axioms (N1)--(N4), except possibly
  axiom (N1)(c), which we must verify directly.  Let therefore $\alpha
  \colon A_1 \to A_2$ be a morphism in $\A_H$.  By applying axiom
  (N1)(c) in $\C$, we can complete this morphism to an $n$-angle
  \begin{equation*}
    A_1 \xrightarrow{\alpha} A_2 \xrightarrow{} C_3 \xrightarrow{}
    \cdots \xrightarrow{} C_n \xrightarrow{} \Sigma A_1
  \end{equation*}
  in $\C$.  Since $\A_H$ is dense in $\C$, there exist objects $C_3',
  \dots, C_n'$ such that $C_i \oplus C_i' \in \A_H$ for $3 \leq i \leq
  n$.  Now take the above $n$-angle and add trivial $n$-angles
  involving the objects $C_i$ and $C_i'$.  The result is an $n$-angle
  \begin{center}
    \begin{tikzpicture}[text centered]
      \node (X1) at (-0.2,0){$A_1$};
      \node (X2) at (1,0){$A_2$};
      \node (X3) at (2.5,0){$C_3 \oplus C_3'$};
      \node (X4) at (4.8,0){$\bigoplus\limits_{i=3}^4 (C_i \oplus
        C_i')$};
      \node (Xdots) at (6.8,0){$\cdots$};
      \node (Xn-1) at (9,0){$\bigoplus\limits_{i=3}^{n-1} (C_i \oplus
        C_i')$};
      \node (Xn) at (11,0){$C$};
      \node (SX1) at (12.3,0){$\Sigma A_1$};
      
      \begin{scope}[font=\scriptsize,->,midway]
        \draw (X1) -- node[above]{$\alpha$} (X2);
        \draw (X2) -- node[above]{} (X3);
        \draw (X3) -- node[above]{} (X4);
        \draw (X4) -- node[above]{} (Xdots);
        \draw (Xdots) -- node[above]{} (Xn-1);
        \draw (Xn-1) -- node[above]{} (Xn);
        \draw (Xn) -- node[above]{} (SX1);
      \end{scope}
    \end{tikzpicture}
  \end{center}
  in $\C$, with $C = C_n \oplus C_{n-1}' \oplus C_{n-2} \oplus \cdots
  \oplus C_4' \oplus C_3$.  In this $n$-angle, all the $n - 1$ first
  vertices are objects in $\A_H$, hence so is $C$ since $\A_H$ is
  complete.  Consequently, this $n$-angle belongs to the collection
  $\nang_{\A_H}$, and we have shown that the collection also satisfies
  axiom (N1)(c).  Therefore $(\A_H,\Sigma)$ is an $n$-angulated
  category.

  It only remains to show that $\A_H$ is an $n$-angulated subcategory
  of $\C$, but this is easy.  The subcategory $\A_H$ is full by
  definition, and it is $n$-angulated and closed under isomorphism by
  the above. Finally, the inclusion functor $\iota \colon \A_H \to \C$
  is, by definition of the collection $\nang_{\A_H}$, an $n$-angulated
  functor.
\end{proof}

We are now ready to prove the main result.  It shows that the
subgroups of the Grothendieck group classify the complete and dense
$n$-angulated subcategories via a bijective correspondence.  As
mentioned after Definition \ref{def:fullsub}, in the triangulated case
every triangulated subcategory is complete, hence we recover
Thomason's classification theorem \cite[Theorem 2.1]{Thomason}.

\begin{theorem}
  \label{thm:correspondence}
  Let $n \geq 3$ be an odd integer and $(\C,\Sigma)$ an $n$-angulated
  category.  Denote the set of complete $n$-angulated subcategories of
  $\C$ by $\complete(\C)$, the set of dense $n$-angulated
  subcategories of $\C$ by $\dense(\C)$, and the set of subgroups of
  $K_0(\C)$ by $\subgroup\left(K_0(\C)\right)$. Then there is a
  one-to-one correspondence
  \begin{equation*}
    \begin{array}{r@{\ }c@{\ }l}
      \complete(\C) \cap \dense(\C) & \longleftrightarrow &
      \subgroup\left(K_0(\C)\right)\\
      \A & \longmapsto & \Image K_0(\A)\\
      \A_H & \longmapsfrom & H
    \end{array}
  \end{equation*}
  where $\A_H$ is the subcategory of $\C$ consisting of those objects
  $A$ in $\C$ such that $[A]\in H \leq K_0(\C)$.
\end{theorem}

\begin{proof}
  By Lemma \ref{lem:subgroup}, for every subgroup $H$ of $K_0(\C)$,
  the category $(\A_H,\Sigma)$ is a complete and dense $n$-angulated
  subcategory of $\C$.  Therefore, the two maps
  \begin{eqnarray*}
    \complete(\C) \cap \dense(\C) & \xrightarrow{\Phi} &
    \subgroup\left(K_0(\C)\right)\\
    \subgroup\left(K_0(\C)\right) & \xrightarrow{\Psi} &
    \complete(\C) \cap \dense(\C)
  \end{eqnarray*}
  given by $\Phi(\A) = \Image K_0(\A)$ and $\Psi(H) = \A_H$ are
  well-defined.  It follows from Lemma \ref{lem:subcategory} that the
  composition $\Psi \circ \Phi$ is the identity on $\complete(\C) \cap
  \dense(\C)$.  For a subgroup $H$ of $K_0(\C)$, it is easy to see
  that $\Phi \circ \Psi(H)$ is a subgroup of $H$, i.e.\ $\Image
  K_0(\A_H)\leq H$.  But every element of $H$ is of the form $[C]$ for
  some $C\in \C$, hence $C \in \A_H$ and then $[C] \in \Image
  K_0(\A_H)$.  This gives $H\leq \Image K_0(\A_H)$, and so the
  composition $\Phi \circ \Psi$ is the identity on
  $\subgroup\left(K_0(\C)\right)$.
\end{proof}

\section{Tensor $n$-angulated categories}
\label{sec:tensor}
In this final section we briefly discuss the classification theorem in
the context of $n$-angulated categories that admit a symmetric tensor
(or monoidal) structure compatible with the $n$-angulated
structure. As in the previous section, we need to restrict to the case
when $n$ is an odd integer. The Grothendieck group of such a category
becomes a commutative ring, whose ideals classify the complete and
dense $n$-angulated tensor ideals.

Recall that an additive category $\C$ is a \emph{symmetric tensor
  category} (or \emph{symmetric monoidal category}) if there is an
additive bifunctor $\otimes \colon \C \times \C \to \C$ and a unit
object $I \in \C$ satisfying the following axioms:
\begin{enumerate}
\item (Associativity Axiom) There is a natural isomorphism $\alpha
  \colon (- \otimes (- \otimes -)) \to ((- \otimes -) \otimes -)$ of
  functors $\C \times \C \times \C \to \C$.
\item (Unit Axiom) There are natural isomorphisms $\lambda \colon (I
  \otimes -) \to (-)$ and $\rho \colon (- \otimes I) \to (-)$ of
  functors $\C \to \C$.
\item (Pentagon Axiom) The diagram
  \begin{center}
    \begin{tikzpicture}[text centered]
      \node (X1) at (-2,0){$A \otimes ((B \otimes C) \otimes D)$};
      \node (X2) at (2,0){$(A \otimes (B \otimes C)) \otimes D$};
      \node (Y1) at (-3,1.5){$A \otimes (B \otimes (C \otimes D))$};
      \node (Y2) at (3,1.5){$((A \otimes B) \otimes C) \otimes D$};
      \node (Z1) at (0,3){$(A \otimes B) \otimes (C \otimes D)$};
      
      \begin{scope}[font=\scriptsize,->,midway]
        \draw (X1) -- node[above]{$\alpha$} (X2);
        \draw (Y1) -- node[below left]{$1 \otimes \alpha$} (X1);
        \draw (X2) -- node[below right]{$\alpha \otimes 1$} (Y2);
        \draw (Y1) -- node[above left]{$\alpha$} (Z1);
        \draw (Z1) -- node[above right]{$\alpha$} (Y2);       
      \end{scope}
    \end{tikzpicture}
  \end{center}
  commutes for all objects $A,B,C,D \in \C$.
\item (Triangle Axiom) The diagram
  \begin{center}
    \begin{tikzpicture}[text centered]
      \node (X1) at (-2,0){$A \otimes (I \otimes B)$};
      \node (X2) at (2,0){$(A \otimes I) \otimes B$};
      \node (Y1) at (0,1.5){$A \otimes B$};
      
      \begin{scope}[font=\scriptsize,->,midway]
        \draw (X1) -- node[above]{$\alpha$} (X2);
        \draw (X1) -- node[above left]{$1 \otimes \lambda$} (Y1);
        \draw (X2) -- node[above right]{$\rho \otimes 1$} (Y1);
      \end{scope}
    \end{tikzpicture}
  \end{center}
  commutes for all objects $A,B \in \C$.
\item (Symmetry Axiom) For all objects $A,B \in \C$ there is an
  isomorphism $\gamma \colon A \otimes B \to B \otimes A$, natural in
  both $A$ and $B$.  The three diagrams
  \begin{center}
    \begin{tikzpicture}[text centered]
      \node (X1) at (-3,0){$A \otimes B$};
      \node (X2) at (0,0){$A \otimes B$};
      \node (X3) at (-1.5,1.5){$B \otimes A$};
      
      \node (Y1) at (2,0){$I \otimes A$};
      \node (Y2) at (5,0){$A$};
      \node (Y3) at (3.5,1.5){$A \otimes I$};

      \node (Z1) at (-2.5,-1.5){$A \otimes (B \otimes C)$};
      \node (Z2) at (1,-1.5){$(A \otimes B) \otimes C$};
      \node (Z3) at (4.5,-1.5){$C \otimes (A \otimes B)$};
      \node (Z4) at (-2.5,-3){$A \otimes (C \otimes B)$};
      \node (Z5) at (1,-3){$(A \otimes C) \otimes B$};
      \node (Z6) at (4.5,-3){$(C \otimes A) \otimes B$};
      
      \begin{scope}[font=\scriptsize,->,midway]
        \draw (X1) -- node[above]{$1$} (X2);
        \draw (X1) -- node[above left]{$\gamma$} (X3);
        \draw (X3) -- node[above right]{$\gamma$} (X2);
        
        \draw (Y1) -- node[above]{$\lambda$} (Y2);
        \draw (Y1) -- node[above left]{$\gamma$} (Y3);
        \draw (Y3) -- node[above right]{$\rho$} (Y2);
      
        \draw (Z1) -- node[above]{$\alpha$} (Z2);
        \draw (Z2) -- node[above]{$\gamma$} (Z3);
        \draw (Z3) -- node[right]{$\alpha$} (Z6);
        \draw (Z1) -- node[left]{$1 \otimes \gamma$} (Z4);
        \draw (Z4) -- node[above]{$\alpha$} (Z5);
        \draw (Z5) -- node[above]{$\gamma \otimes 1$} (Z6);
      \end{scope}
    \end{tikzpicture}
  \end{center}
  commute for all objects $A,B,C \in \C$.
\end{enumerate}

For further details, we refer to \cite[VII.1 and VII.7]{MacLane}.
Strictly speaking, one should refer to the symmetric tensor category
$\C$ as the tuple $(\C, \otimes, I, \alpha, \lambda, \rho, \gamma)$.
However, we shall keep the notation at a minimum and only refer to
``the symmetric tensor category $\C$.''

Now let $(\C,\Sigma)$ be an $n$-angulated category.  Then $\C$ is
\emph{tensor $n$-angulated} if it admits a symmetric tensor structure
which is compatible with the $n$-angulated structure.  Specifically,
this means that $\C$ is also a symmetric tensor category $(\C,
\otimes, I, \alpha, \lambda, \rho, \gamma)$ satisfying the following
axioms:
\begin{enumerate}
\item There are natural isomorphisms $l \colon (- \otimes ( \Sigma -))
  \to \Sigma (- \otimes -)$ and $r \colon ( ( \Sigma -) \otimes -) \to
  \Sigma (- \otimes -)$ of functors $\C \times \C \to \C$.
\item For every object $A \in \C$, the endofunctors $(A \otimes -)$
  and $(- \otimes A)$ on $\C$ are $n$-angulated functors, together
  with the natural isomorphisms $l$ and $r$, respectively.
\item The two diagrams
  \begin{center}
    \begin{tikzpicture}[text centered]
      \node (X1) at (-3,0){$I \otimes \Sigma A$};
      \node (X2) at (0,0){$\Sigma A$};
      \node (X3) at (-1.5,1.5){$\Sigma (I \otimes A)$};
      
      \node (Y1) at (2,0){$\Sigma A \otimes I$};
      \node (Y2) at (5,0){$\Sigma A$};
      \node (Y3) at (3.5,1.5){$\Sigma (A \otimes I)$};
      
      \begin{scope}[font=\scriptsize,->,midway]
        \draw (X1) -- node[above]{$\lambda$} (X2);
        \draw (X1) -- node[above left]{$l$} (X3);
        \draw (X3) -- node[above right]{$\Sigma \lambda$} (X2);
        
        \draw (Y1) -- node[above]{$\rho$} (Y2);
        \draw (Y1) -- node[above left]{$r$} (Y3);
        \draw (Y3) -- node[above right]{$\Sigma \rho$} (Y2);
        \end{scope}
    \end{tikzpicture}
  \end{center}
  commute for every object $A \in \C$.
\item The diagram
  \begin{center}
    \begin{tikzpicture}[text centered]
      \node (X1) at (-3,1.5){$\Sigma A \otimes \Sigma B$};
      \node (X2) at (0,1.5){$\Sigma (A \otimes \Sigma B)$};
      \node (X3) at (-3,0){$\Sigma ( \Sigma A \otimes B)$};
      \node (X4) at (0,0){$\Sigma^2 (A \otimes B)$};
      
      \begin{scope}[font=\scriptsize,->,midway]
        \draw (X1) -- node[above]{$r$} (X2);
        \draw (X1) -- node[left]{$l$} (X3);
        \draw (X2) -- node[right]{$\Sigma l$} (X4);
        \draw (X3) -- node[above]{$\Sigma r$} (X4);
        \end{scope}
    \end{tikzpicture}
  \end{center}
  anti-commutes for all objects $A,B \in \C$.
\end{enumerate}

Note that axiom (2) can be reformulated as follows: for every object
$A$ and $n$-angle
\begin{equation*}
  A_1 \xrightarrow{\alpha_1} A_2 \xrightarrow{\alpha_2} \cdots
  \xrightarrow{\alpha_{n - 1}} A_n \xrightarrow{\alpha_n} \Sigma A_1 
\end{equation*}
in $\C$, the two $n$-$\Sigma$-sequences
\begin{equation*}
  A \otimes A_1 \xrightarrow{1 \otimes \alpha_1} A \otimes A_2
  \xrightarrow{1 \otimes \alpha_2} \cdots \xrightarrow{1 \otimes
    \alpha_{n - 1}} A \otimes A_n \xrightarrow{l \circ (1 \otimes
    \alpha_n)} \Sigma (A \otimes A_1)
\end{equation*}
\begin{equation*}
  A_1 \otimes A \xrightarrow{\alpha_1 \otimes 1} A_2 \otimes A
  \xrightarrow{\alpha_2 \otimes 1} \cdots \xrightarrow{\alpha_{n - 1}
    \otimes 1} A_n \otimes A \xrightarrow{r \circ (\alpha_n \otimes
    1)} \Sigma (A_1 \otimes A)
\end{equation*}
are also $n$-angles in $\C$.  Note also that in the triangulated case,
i.e.\ when $n = 3$, then some authors include further axioms for a
tensor triangulated category, cf.\ \cite[Remark 4]{Balmer} and
\cite{HoveyPalmieriStrickland, KellerNeeman, May}.  However, we will
not need $n$-angulated versions of these axioms.

\begin{lemma}
  \label{lem:ring}
  Let $(\C,\Sigma)$ be a tensor $n$-angulated category, with $n$ an
  odd integer.  Then the Grothendieck group $K_0(\C)$ is a commutative
  ring with multiplication given by $[A][B] = [A \otimes B]$ for
  objects $A,B \in \C$.
\end{lemma}

\begin{proof}
  Suppose that $A,B$ and $B'$ are objects in $\C$ with $[B] = [B']$ in
  $K_0(\C)$.  We first have to show that the multiplication is
  well-defined, i.e.\ that $[A][B] = [A][B']$.  By Corollary
  \ref{cor:equality}, there exist objects $C_1, \dots, C_n$ and two
  $n$-angles
  \begin{equation*}
    B \oplus C_1 \xrightarrow{\alpha_1} C_2 \xrightarrow{\alpha_2}
    \cdots \xrightarrow{\alpha_{n-1}} C_n \xrightarrow{\alpha_n}
    \Sigma A \oplus \Sigma C_1
  \end{equation*}
  \begin{equation*}
    B' \oplus C_1 \xrightarrow{\beta_1} C_2 \xrightarrow{\beta_2}
    \cdots \xrightarrow{\beta_{n-1}} C_n \xrightarrow{\beta_n}
    \Sigma B \oplus \Sigma C_1
  \end{equation*}
  in $\C$.  Tensoring these with the object $A$ yields two new
  $n$-angles, and together with Corollary \ref{cor:equality} again
  these show that
  \begin{equation*}
    [A][B] = [A \otimes B] = [A \otimes B'] = [A][B']
  \end{equation*}
  in $K_0(\C)$.  The multiplication is therefore well-defined.

  The associativity axiom and symmetry axiom for a symmetric tensor
  category guarantees that the multiplication in $K_0(\C)$ is
  associative and commutative.  Furthermore, by the unit axiom, the
  image $[I]$ in $K_0(\C)$ of the tensor unit object $I \in \C$ is the
  multiplicative identity.  Finally, the equalities
  \begin{equation*}
    [A] \left([B] + [C]\right) = [A][B \oplus C] = [A \otimes (B
    \oplus C)] = [(A \otimes B) \oplus (A \otimes C)] = [A \otimes B]
    + [A \otimes C] = [A][B] + [A][C]
  \end{equation*}
  show that the distributive law holds.
\end{proof}

We may therefore speak of the \emph{Grothendieck ring} $K_0(\C)$ of a
tensor $n$-angulated category $(\C,\Sigma)$.  We end with a tensor
version of Theorem \ref{thm:correspondence}, showing that the set of
ideals in this ring classifies the ``ideals'' in $\C$.  Namely, let
$\A$ be a subcategory of $\C$.  Then $\A$ is an \emph{$n$-angulated
  tensor ideal} of $\C$ if it is an $n$-angulated subcategory with the
following additional property: for objects $A \in \A$ and $C \in \C$
the tensor product $C \otimes A$ belongs to $\A$.  Furthermore, we say
that $\A$ is an \emph{$n$-angulated tensor prime ideal} of $\C$ if it
is an $n$-angulated tensor ideal with the following property: if $C
\otimes C'$ belongs to $\A$ for objects $C,C' \in \C$, then either $C$
or $C'$ belongs to $\A$.

\begin{theorem}
  \label{thm:tensorcorrespondence}
  Let $n \geq 3$ be an odd integer and $(\C,\Sigma)$ a tensor
  $n$-angulated category.  Denote the set of complete $n$-angulated
  tensor ideals of $\C$ by $\complete^{\otimes}(\C)$, the set of dense
  $n$-angulated tensor ideals of $\C$ by $\dense^{\otimes}(\C)$, and
  the set of ideals of $K_0(\C)$ by $\ideal\left(K_0(\C)\right)$.
  Furthermore, denote the set of complete $n$-angulated tensor prime
  ideals of $\C$ by $\complete^{\otimes}_p(\C)$, the set of dense
  $n$-angulated tensor prime ideals of $\C$ by
  $\dense^{\otimes}_p(\C)$, and the set of prime ideals of $K_0(\C)$
  by $\Prime\left(K_0(\C)\right)$.  Then there are one-to-one
  correspondences
  \begin{equation*}
    \begin{array}{r@{\ }c@{\ }l}
      \complete^{\otimes}(\C) \cap \dense^{\otimes}(\C) &
      \longleftrightarrow & \ideal\left(K_0(\C)\right)\\
      \complete^{\otimes}_p(\C) \cap \dense^{\otimes}_p(\C) &
      \longleftrightarrow & \Prime\left(K_0(\C)\right)\\
      \A & \longmapsto & \Image K_0(\A)\\
      \A_H & \longmapsfrom & H
    \end{array}
  \end{equation*}
  where $\A_H$ is the subcategory of $\C$ consisting of those objects
  $A$ in $\C$ such that $[A]\in H \leq K_0(\C)$.
\end{theorem}

\begin{proof}
  In view of Theorem \ref{thm:correspondence}, it suffices to show
  that the given correspondences map complete and dense $n$-angulated
  tensor (prime) ideals of $\C$ to (prime) ideals of the Grothendieck
  ring, and vice versa.

  Let $H$ be an ideal in $K_0(\C)$, and consider the subcategory
  $\A_H$ of $\C$.  We know from Theorem \ref{thm:correspondence} that
  $\A_H$ is a complete and dense $n$-angulated subcategory of $\C$.
  If $A \in \A_H$ and $C \in \C$, then by definition $[A]$ belongs to
  $H$ in $K_0(\C)$.  Since $H$ is an ideal, the element $[C][A]$ also
  belongs to $H$, i.e.\ $[C \otimes A] \in H$.  But then $C \otimes A$
  belongs to $\A_H$, hence $\A_H$ is an $n$-angulated tensor ideal of
  $\C$.  Suppose now that $H$ is a prime ideal, and let $C,C' \in \C$
  be objects with $C \otimes C' \in \A_H$.  Then $[C \otimes C']$
  belongs to $H$ in $K_0(\C)$, i.e.\ $[C][C'] \in H$.  Since $H$ is a
  prime ideal, either $[C]$ or $[C']$ belongs to $H$, hence either $C$
  or $C'$ belongs to $\A_H$.  This shows that $\A_H$ is an
  $n$-angulated tensor prime ideal of $\C$.

  Conversely, let $\A$ be a complete and dense $n$-angulated tensor
  ideal of $\C$, and consider the subgroup $\Image K_0(\A)$ of
  $K_0(\C)$.  Let $[C]$ and $[A]$ be elements in $K_0(\C)$ with
  $[A] \in \Image K_0(\A)$.  Using Theorem \ref{thm:correspondence}
  again, we see that the object $A$ then belongs to $\A$.  Since $\A$
  is an $n$-angulated tensor ideal of $\C$, the object $C \otimes A$
  also belongs to $\A$, hence $[C \otimes A] \in \Image K_0(\A)$ in
  $K_0(\C)$. But then $[C][A] \in \Image K_0(\A)$, hence $\Image
  K_0(\A)$ is an ideal of $K_0(\C)$.  Suppose now that $\A$ is an
  $n$-angulated tensor prime ideal of $\C$, and let $[C],[C'] \in
  K_0(\C)$ be elements with $[C][C'] \in \Image K_0(\A)$.  Then $[C
  \otimes C']$ belongs to $\Image K_0(\A)$, and as above this implies
  that the object $C \otimes C'$ belongs to $\A$.  Since $\A$ is an
  $n$-angulated tensor prime ideal of $\C$, either $C$ or $C'$ belongs
  to $\A$, hence either $[C]$ or $[C']$ belongs to $\Image K_0(\A)$ in
  $K_0(\C)$.  This shows that $\Image K_0(\A)$ is a prime ideal of
  $K_0(\C)$.
\end{proof}

\bibliographystyle{plain}

\begin{thebibliography}{GKO}
\bibitem[Bal]{Balmer}P.\ Balmer, \emph{Tensor triangular geometry},
  Proceedings of the International Congress of Mathematicians, Volume
  II, 85--112, Hindustan Book Agency, New Delhi, 2010.
\bibitem[BeT]{BerghThaule}P.A.\ Bergh, M.\ Thaule, \emph{The
    axioms for $n$-angulated categories}, preprint (2011),
  arXiv:1112.2533v2.
\bibitem[GKO]{GeissKellerOppermann}C.\ Geiss, B.\ Keller, S.\
  Oppermann, \emph{$n$-angulated categories}, to appear in J.\ Reine
  Angew.\ Math.
\bibitem[Gro]{Grothendieck}A.\ Grothendieck, \emph{Groupes de classes
    des cat\'egories ab\'eliennes et triangul\'ees, Complexes
    parfaits}, SGA 5, Expos\'e VIII, Springer LNM \textbf{589} (1977),
  351--371.
\bibitem[HPS]{HoveyPalmieriStrickland}M.\ Hovey, J.H.\ Palmieri, N.H.\
  Strickland, \emph{Axiomatic stable homotopy theory}, Mem.\ Amer.\
  Math.\ Soc.\ \textbf{128} (1997), no.\ 610, x+114 pp.
\bibitem[KeN]{KellerNeeman}B.\ Keller, A.\ Neeman, \emph{The
    connection between May's axioms for a triangulated tensor product
    and Happel's description of the derived category of the quiver
    $D_4$}, Doc.\ Math.\ \textbf{7} (2002), 535--560.
\bibitem[KeR]{KellerReiten}B.\ Keller, I.\ Reiten,
  \emph{Cluster-tilted algebras are Gorenstein and stably Calabi-Yau},
  Adv.\ Math.\ \textbf{211} (2007), no.\ 1, 123--151.
\bibitem[Lan]{Landsburg}S.E.\ Landsburg, \emph{$K$-theory and patching
    for categories of complexes}, Duke Math.\ J.\ \textbf{62} (1991),
  no.\ 2, 359--384.
\bibitem[Mac]{MacLane}S.\ Mac Lane, \emph{Categories for the working
    mathematician}, second edition, Graduate Texts in Mathematics
  \textbf{5}, Springer-Verlag, New York, 1998, xii+314 pp.
\bibitem[May]{May}J.P.\ May, \emph{The additivity of traces in
    triangulated categories}, Adv.\ Math.\ \textbf{163} (2001), no.\
  1, 34--73.
\bibitem[Tho]{Thomason}R.W.\ Thomason, \emph{The classification of
    triangulated subcategories}, Compositio Math.\ \textbf{105}
  (1997), no.\ 1, 1--27.
\bibitem[Ver]{Verdier} J.-L.\ Verdier, \emph{Cat\'egories
    d\'eriv\'ees}, SGA $4\frac12$: Cohomologie Etale, Springer LNM
  \textbf{569} (1977), 262--311.
\end{thebibliography}

\end{document}